\documentclass{amsart}
\usepackage{graphicx}
\usepackage{amsmath}
\usepackage{amssymb}
\usepackage{amsthm}

\usepackage{endnotes}
\usepackage{enumerate}
\usepackage[numbers]{natbib}

\newtheorem{prob}{Problem}
\newtheorem{ques}{Question}

\newtheorem{thm}{Theorem}[section]

\newtheorem{lem}[thm]{Lemma}
\newtheorem{cor}[thm]{Corollary}

\theoremstyle{definition}
\newtheorem{defn}{Definition}

\def\PA{\mathsf{PA}}
\def\TA{\mathsf{TA}}

\def\cee{\mathsf{C}}

\DeclareMathOperator{\Lt}{Lt}

\DeclareMathOperator{\scl}{Scl}
\DeclareMathOperator{\Def}{Def}

\begin{document}

\title{Enayat Models of Peano Arithmetic}
\author{Athar Abdul-Quader}
\date{September 21, 2017}

\begin{abstract}
Simpson \cite{simpson} showed that every countable model $\mathcal{M} \models \PA$ has an expansion $(\mathcal{M}, X) \models \PA^*$ that is pointwise definable. A natural question is whether, in general, one can obtain expansions of a non-prime model in which the definable elements coincide with those of the underlying model. Enayat \cite{enayat1988undefinable} showed that this is impossible by proving that there is $\mathcal{M} \models \PA$ such that for each undefinable class $X$ of $\mathcal{M}$, the expansion $(\mathcal{M}, X)$ is pointwise definable. We call models with this property Enayat models. In this paper, we study Enayat models and show that a model of $\PA$ is Enayat if it is countable, has no proper cofinal submodels and is a conservative extension of all of its elementary cuts. We then show that, for any countable linear order $\gamma$, if there is a model $\mathcal{M}$ such that $\Lt(\mathcal{M}) \cong \gamma$, then there is an Enayat model $\mathcal{M}$ such that $\Lt(\mathcal{M}) \cong \gamma$.
\end{abstract}

\maketitle

\section{Introduction}

Given a model $\mathcal{M}$ of $\PA$, a subset $X \subseteq M$ is called \emph{inductive} if $(\mathcal{M}, X) \models \PA^*$. In other words, $X$ is inductive if the structure $(\mathcal{M}, X)$ satisfies the induction schema for all formulas in the expanded language with a predicate symbol for $X$. A set $X \subseteq M$ is called a \emph{class} if, for each $a \in M$, $\{ x \in X : x < a \} \in \Def(\mathcal{M})$; that is, $X$ is a class if every initial segment of $X$ is definable with parameters in $\mathcal{M}$. Every inductive subset of a model of $\PA$ is a class. Simpson \cite{simpson} showed that every countable model $\mathcal{M} \models \PA$ has a pointwise definable expansion $(\mathcal{M}, X) \models \PA^*$, where $X \subseteq M$. Simpson's argument uses arithmetic forcing, which produces an undefinable, inductive set $X \subseteq M$. One may ask whether arithmetic forcing can be used to find an undefinable, inductive set $X \subseteq M$ so that no new elements are definable in $(\mathcal{M}, X)$. Enayat \cite{enayat1988undefinable} showed that this is impossible: for every completion $T$ of $\PA$, there are $2^{\aleph_0}$ non-isomorphic models $\mathcal{M} \models T$ with the property that for any undefinable class $X \subseteq M$, the expansion $(\mathcal{M}, X)$ is pointwise definable. Enayat's result inspires the following definition:

\begin{defn}
Let $\mathcal{M} \models \PA$ be countable. If $\mathcal{M}$ is not prime and, for every undefinable class $X$ of $M$, $(\mathcal{M}, X)$ is pointwise definable, then $\mathcal{M}$ is called an \emph{Enayat model}.
\end{defn}

If $\mathcal{M} \prec \mathcal{N}$, we say that $\mathcal{N}$ is a \emph{minimal} extension of $\mathcal{M}$ if whenever $\mathcal{M} \preccurlyeq \mathcal{K} \preccurlyeq \mathcal{N}$, then either $\mathcal{K} = \mathcal{M}$ or $\mathcal{K} = \mathcal{N}$. Given a model $\mathcal{M} \models \PA$ and a set $X \subseteq M$, the Skolem closure of $X$, denoted $\scl^\mathcal{M}(X)$ is the smallest elementary submodel of $\mathcal{M}$ containing $X$. We often suppress the reference to the larger model $\mathcal{M}$ and write $\scl(X)$. An elementary extension $\mathcal{M} \prec \mathcal{N}$ is called superminimal if, for all $a \in N \setminus M$, $\mathcal{N} = \scl(a)$; it is clear that superminimal extensions are also minimal extensions. An extension $\mathcal{M} \prec \mathcal{N}$ is conservative, denoted $\mathcal{M} \prec_\text{cons} \mathcal{N}$ if, for all $X \in \Def(\mathcal{N})$, $X \cap M \in \Def(\mathcal{M})$. Enayat \cite{enayat1988undefinable} showed that, for each completion $T$ of $\PA$, any minimal conservative extension of the prime model of $T$ is Enayat. By a similar proof if $\alpha$ is a countable ordinal, then the union of an elementary chain of superminimal conservative extensions of length $\alpha$ is Enayat. Such models exist because every countable model of $\PA$ has a superminimal conservative extension (\cite[Corollary 2.2.12]{tsmopa}).

The work in this paper is based in large part on the discussion of substructure lattices of models of $\PA$ given in \cite[Chapter 4]{tsmopa}. We will repeat some definitions and results here.

Given $\mathcal{M} \models \PA$, the set of all $\mathcal{K} \prec \mathcal{M}$ forms a lattice under inclusion, called the \emph{substructure lattice} of $\mathcal{M}$ and denoted $\Lt(\mathcal{M})$. Given $\mathcal{M} \prec \mathcal{N}$, the \textit{interstructure lattice}, denoted $\Lt(\mathcal{N} / \mathcal{M})$ is the set of all $\mathcal{K}$ such that $\mathcal{M} \preceq \mathcal{K} \preceq \mathcal{N}$. Given a lattice $L$, $a \in L$ is \textit{compact} if whenever $X \subseteq L$ and $a \leq \bigvee X$, then there is a finite $Y \subseteq X$ such that $a \leq \bigvee Y$. $L$ is \textit{algebraic} if it is complete and each $a \in L$ is a supremum of a set of compact elements. If $\kappa$ is a cardinal, then $L$ is $\kappa$-algebraic if it is algebraic and each $a \in L$ has less than $\kappa$ compact predecessors. If $\mathcal{M} \models \PA$, then $\Lt(\mathcal{M})$ is $\aleph_1$-algebraic.

Section \ref{main} of this paper characterizes which finite lattices can be realized as the substructure lattice of an Enayat model. Section \ref{infdesc} contains the first main result of this paper, Theorem \ref{cuts}, which states that a countable model of $\PA$ that is a conservative extension of all its submodels, and contains no proper cofinal submodel is Enayat. Section \ref{mainsect} contains the second main result, Theorem \ref{mainthm}, which shows that any countable linear order that can be the substructure lattice of a model of $\PA$ can be the substructure lattice of an Enayat model. We conclude with some open problems in Section \ref{opens}.

This paper grew out of work that appeared as a chapter in the author's Ph.D. thesis.\endnote{The author is grateful to his advisor Roman Kossak, and supervisory committee members Alfred Dolich, Russell Miller, and Philipp Rothmaler for their many helpful revisions.}\endnote{The author is grateful to Jim Schmerl for his comments which led to significant improvements throughout this paper, in particular in Theorems \ref{cuts} and \ref{mainthm}, where his comments helped simplify the proofs significantly.}

\section{Enayat Models With Finite Substructure Lattices}\label{main}

The ultimate goal of this project is to give a complete characterization of Enayat models in terms of better-known model-theoretic properties. So far, we can identify a few such properties. First we show that Enayat models cannot have proper cofinal submodels.

\begin{lem}\label{cofinality}
Let $\mathcal{M} \models \PA$ be countable and suppose $\mathcal{K} \prec_\text{cof} \mathcal{M}$ is a proper submodel. Then $\mathcal{M}$ is not Enayat.
\end{lem}

\begin{proof}
Because $\mathcal{K}$ is countable, we can find an undefinable inductive subset $X$ of $\mathcal{K}$ by arithmetic forcing. We can extend this $X$ to $Y \subseteq M$ as follows: for each $a \in K$ we there is some formula $\phi_a(x)$ (possibly using parameters from $K$) which defines 
\begin{displaymath}\{ x \in K : (\mathcal{K}, X) \models x \leq a \wedge x \in X \} 
\end{displaymath}

Then let $Y = \bigcup\limits_{a \in K} \{ x \in M : \mathcal{M} \models \phi_a(x) \}$, and one can show that $(\mathcal{K}, X) \prec (\mathcal{M}, Y)$. This result is due independently to Kotlarski and Schmerl; see \cite[Theorem 1.3.7]{tsmopa}. Since $\scl^{(\mathcal{M}, Y)}(0) \subseteq \mathcal{K}$, $\mathcal{M}$ is not Enayat.
\end{proof}

Lemma \ref{cofinality} gives us an easy characterization of which finite lattices can appear as the substructure lattices of an Enayat model. To state this characterization, we recall the ``lattice sum'' notation. Given two lattices $L_1$ and $L_2$, if $L_1$ has a top element and $L_2$ has a bottom element, then the lattice $L = L_1 \oplus L_2$ is the lattice formed by identifying the top element of $L_1$ with the bottom element of $L_2$. In particular, for any lattice $L$, $L \oplus \mathbf{2}$ is the lattice formed by adding one new element above the top element of $L$. As an example, if $\mathcal{N}$ is a superminimal elementary extension of $\mathcal{M}$, then $\Lt(\mathcal{N}) \cong \Lt(\mathcal{M}) \oplus \mathbf{2}$.

\begin{cor}\label{enayatlattices}
\ \begin{enumerate}
\item Let $\mathcal{M} \models \PA$ be an Enayat model. If $\Lt(\mathcal{M})$ is finite, then it is of the form $L \oplus \bf{2}$ where $L$ is some finite lattice.\label{enayatlplus2}
\item Let $L$ be a finite lattice, $T$ a completion of $\PA$ and $T \neq \TA$. If there is $\mathcal{N} \models T$ such that $\Lt(\mathcal{N}) \cong L$, then there is an Enayat $\mathcal{M} \models T$ such that $\Lt(\mathcal{M}) \cong L \oplus \bf{2}$.\label{lplus2enayat}
\end{enumerate}
\end{cor}
\begin{proof}
To prove (\ref{enayatlplus2}), all we need to show here is that the top element of $\Lt(\mathcal{M})$ cannot have more than one immediate predecessor. Suppose there are two: $\mathcal{K}_1$ and $\mathcal{K}_2$. Notice that, since these are immediate predecessors of $\mathcal{M}$, the extensions $\mathcal{K}_i \prec \mathcal{M}$ are minimal. By Gaifman's Splitting Theorem (\cite{gaifmansplit}), there is $\bar{\mathcal{K}_i}$ such that $\mathcal{K}_i \preceq_\text{cof} \bar{\mathcal{K}_i} \preceq_\text{end} \mathcal{M}$. By minimality, for each $i$, either $\mathcal{K}_i = \bar{\mathcal{K}_i}$ or $\bar{\mathcal{K}_i} = \mathcal{M}$. So either $\mathcal{K}_i \prec_\text{end} \mathcal{M}$ or $\mathcal{K}_i \prec_\text{cof} \mathcal{M}$. Suppose neither $\mathcal{K}_1$ nor $\mathcal{K}_2$ is cofinal in $\mathcal{M}$, and therefore they are both cuts. Because $\mathcal{K}_1$ and $\mathcal{K}_2$ are incomparable in $\Lt(\mathcal{M})$, there are $a \in K_1 \setminus K_2$ and $b \in K_2 \setminus K_1$. Then either $\mathcal{M} \models a < b$ or $\mathcal{M} \models b < a$. Because the $\mathcal{K}_i$ are cuts, in the former case, that means $a \in K_2$ and in the latter case, $b \in K_1$. These are both contradictions, so one of the $\mathcal{K}_i$ must be a cofinal submodel of $\mathcal{M}$. This is impossible if $\mathcal{M}$ is Enayat by Lemma \ref{cofinality}.

For the proof of (\ref{lplus2enayat}), let $\mathcal{M}_T \models T$ be a prime model of $T$. Since there is $\mathcal{N} \models T$ with $\Lt(\mathcal{N}) \cong L$, then by Theorems 4.5.21 and 4.5.22 in \cite{tsmopa}, there is a cofinal extension $\mathcal{K}$ of $\mathcal{M}_T$ such that $\Lt(\mathcal{K}) \cong L$. Let $\mathcal{M}$ be a superminimal conservative extension of $\mathcal{K}$. Theorem 2.2.13 in \cite{tsmopa} shows that this $\mathcal{M}$ must be Enayat.
\end{proof}

Corollary \ref{enayatlattices} characterizes the finite lattices which can appear as the substructure lattice of an Enayat model. To see this, we note that if $L$ is a finite lattice that is the substructure lattice of a model of $\TA$, then, for any completion $T \neq \TA$ of $\PA$, there is a cofinal extension $\mathcal{M}$ of the prime model $\mathcal{M}_T$ such that $\Lt(\mathcal{M}) \cong L$. To get such an extension, we appeal to Theorems 4.5.21 and 4.5.22 in \cite{tsmopa}.

The following remains open:

\begin{ques}
Which finite lattices can be realized as the substructure lattice of an Enayat model of $\TA$?
\end{ques}

We can modify the proof of Corollary \ref{enayatlattices}(\ref{lplus2enayat}) to get that, for a finite lattice $L$, if there is a model $\mathcal{M} \models \TA$ such that $\Lt(\mathcal{M}) \cong \mathbf{2} \oplus L$, then there is an Enayat model of $\TA$ whose substructure lattice is $\mathbf{2} \oplus L \oplus \mathbf{2}$. This is done in much the same way: first we find a minimal, conservative extension $\mathcal{M}$ of $\mathbb{N}$, and then find a cofinal extension $\mathcal{M}_1$ of $\mathcal{M}$ such that $\Lt(\mathcal{M}_1) \cong \mathbf{2} \oplus L$ and so that the greatest common initial segment between $\mathcal{M}$ and $\mathcal{M}_1$ contains a non-standard element. Then a superminimal conservative extension of $\mathcal{M}_1$ is Enayat.

Other Enayat models of $\TA$ can be found using results in the next section. As an example, there is an Enayat model of $\TA$ whose substructure lattice is isomorphic to $\mathbf{B}_2 \oplus \mathbf{2}$, showing that substructure lattices of models of $\TA$ need not be isomorphic to a lattice of the form $\mathbf{2} \oplus L \oplus \mathbf{2}$ for some finite lattice $L$. To find such a model, let $p(x)$ be a minimal type over $\TA$ and let $a$ and $b$ be two elements realizing it. Then if $\mathcal{M}$ is a superminimal conservative extension of $\scl(a, b)$, it is Enayat and $\Lt(\mathcal{M}) \cong \mathbf{B}_2 \oplus \mathbf{2}$.

Corollary \ref{enayatlattices} implies that there are Enayat models of $\PA$ whose substructure lattice is isomorphic to $\mathbf{N}_5 \oplus \mathbf{2}$. It is unknown whether there is an Enayat model of $\TA$ whose substructure lattice is isomorphic to this lattice; more generally, it is unknown if there are Enayat models which are not conservative over all their elementary cuts. If $\mathcal{M} \models \TA$ is such that $\Lt(\mathcal{M}) \cong \mathbf{N}_5 \oplus \mathbf{2}$, then $\mathcal{M}$ is not a conservative extension of $\mathbb{N}$.

\section{Characterizing Enayat Models}\label{infdesc}

In this section, we show our first main result: a model is Enayat if it has no proper cofinal submodel and is a conservative extension of each of its elementary cuts. First, we prove a lemma which will be needed for this result. This lemma is very similar to \cite[Theorem 2.2.13]{tsmopa}.

\begin{lem}\label{limit}
Suppose $\mathcal{N} \models \PA$, $X$ is an undefinable class of $\mathcal{N}$, $\mathcal{M} \prec_\text{cons} \mathcal{N}$, and $C$ is a cofinal subset of $\mathcal{M}$. Then there is $b \in N \setminus M$ such that $b \in \mathrm{dcl}^{(\mathcal{N}, X)}(C)$.
\end{lem}

\begin{proof}
Because $\mathcal{N}$ is a conservative extension of $\mathcal{M}$, we have, for some $b \in M$:
\begin{displaymath}
X \cap M = \{ x \in M : \mathcal{M} \models \phi(x, b) \}
\end{displaymath} Because $C$ is cofinal in $\mathcal{M}$, there is some $a \in C$ such that $b < a$. Consider the set 
\begin{displaymath}
Y = \{ z \in N : (\mathcal{N}, X) \models \exists y < a \: \forall x < z \: (\phi(x, y) \leftrightarrow x \in X) \}
\end{displaymath}

This set contains $\mathcal{M}$. It must also be bounded, since, if it were not, then $Y = N$, and there would be some $b < a$ such that
\begin{displaymath} X = \{ x \in N : \mathcal{N} \models \phi(x, b) \}
\end{displaymath}
However, since $X$ is undefinable, there can be no such $b$. Let $b$ be the maximum of $Y$. Clearly $b$ is a definable element in $(\mathcal{N}, X)$ using only parameters from $C$, and is above $\mathcal{M}$.
\end{proof}

\begin{thm}\label{cuts}
Suppose $\mathcal{M}$ is countable, has no proper cofinal submodel, and is a conservative extension of each of its elementary cuts. Then $\mathcal{M}$ is Enayat.
\end{thm}

\begin{proof}
Let $X \subseteq M$ be an undefinable class and let $C$ be the set of all elements definable in $(\mathcal{M}, X)$. Let $\mathcal{K} = \mathrm{sup}(C)$. Then $\mathcal{K} \prec_\text{cons} \mathcal{M}$. If $C$ is bounded in $\mathcal{M}$, there is $c \in M \setminus K$ definable in $(\mathcal{M}, X)$, which is a contradiction; therefore $C$ must be cofinal in $\mathcal{M}$. Since $\mathcal{M}$ has no proper cofinal submodels, $C = M$.
\end{proof}

We can find many examples of Enayat models as a result of this theorem. As mentioned before, Corollary 2.2.12 of \cite{tsmopa} states that every countable model of $\PA$ has a superminimal conservative extension. This means we form countable elementary chains of superminimal conservative extensions, which, by Theorem \ref{cuts}, are Enayat models. That is, if $\alpha$ is a countable ordinal, $\mathcal{N} = \bigcup\limits_{\beta < \alpha} \mathcal{M}_\beta$, where $\mathcal{M}_0$ is prime, $\mathcal{M}_{\beta + 1}$ is a superminimal conservative extension of $\mathcal{M}_\beta$, and $\mathcal{M}_\lambda = \bigcup\limits_{\beta < \lambda} \mathcal{M}_\beta$ whenever $\lambda$ is a limit ordinal, then $\mathcal{N}$ is Enayat. 

Corollary \ref{enayatlattices} characterized the finite lattices which can be the substructure lattices of an Enayat model. For infinite lattices, we do not have a complete characterization; however, we note the following corollary of Theorem \ref{cuts}.

\begin{cor}\label{countablelattices}
Let $T \neq \TA$ be a completion of $\PA$, $\mathcal{M}_T \models T$ a prime model of $T$, and $L$ a lattice. Suppose there is a countable $\mathcal{N} \succ_\text{cof} \mathcal{M}_T$ such that $\Lt(\mathcal{N}) \cong L$. Then there is an Enayat model $\mathcal{M} \models T$ such that $\Lt(\mathcal{M}) \cong L \oplus \mathbf{2}$.
\end{cor}

\begin{proof}
Let $\mathcal{M}$ be a superminimal conservative extension of $\mathcal{N}$. $\mathcal{N}$ is the only proper elementary cut of $\mathcal{M}$ and $\mathcal{M}$ has no proper cofinal submodels. By Theorem \ref{cuts}, $\mathcal{M}$ is Enayat.
\end{proof}

Many examples of lattices can be realized as substructure lattices of Enayat models in this way. Paris \cite{parisdistributive} proved if $L$ is a countable algebraic distributive lattice, then for any completion $T$ of $\PA$, there is $\mathcal{M} \models T$ with $\Lt(\mathcal{M}) \cong L$. This proof can be modified (see \cite[Theorem 4.7.3]{tsmopa}) to obtain the following: if $L$ is a countable algebraic distributive lattice, then every countable nonstandard $\mathcal{M} \models \PA$ has a cofinal extension $\mathcal{N}$ such that $\Lt(\mathcal{N} / \mathcal{M}) \cong L$.

\section{Linearly Ordered Substructure Lattices}\label{mainsect}
In this section, we construct Enayat models which are conservative extensions of all of their submodels; the substructure lattices of these models are linear orders. We show a more general statement: if $\gamma$ is a linear order such that there is some model of $\PA$ whose substructure lattice is isomorphic to $\gamma$, then there is a model $\mathcal{M} \models \PA$ whose substructure lattice is isomorphic to $\gamma$ with the property that for each $\mathcal{K} \in \Lt(\mathcal{M})$, $\mathcal{M}$ is a conservative elementary extension of $\mathcal{K}$. If, in addition, $\gamma$ is countable, then such an $\mathcal{M}$ is Enayat.

\begin{thm}\label{mainthm}
Let $T$ be a completion of $\PA$ and let $\gamma$ be an $\aleph_1$-algebraic algebraic linear order (that is, $\gamma$ is complete, compactly generated, and has countably many compact elements). There is $\mathcal{M} \models T$ such that $\Lt(\mathcal{M}) \cong \gamma$ and, for each $\mathcal{K} \in \Lt(\mathcal{M})$, $\mathcal{K} \prec_\text{cons} \mathcal{M}$.
\end{thm}

\begin{proof}
The \emph{cofinality} quantifier $\cee$ is defined so that $\cee x$ is shorthand for $\forall w \exists x > w$. It is understood that the variable $w$ does not appear elsewhere. The \emph{cobounded} quantifier $\cee^*$ is the dual of $\cee$; that is, $\cee^*$ is $\lnot \cee \lnot$. It can be thought of as shorthand for $\exists w \forall x > w$.

We extend $\cee$ and $\cee^*$ to apply to $n$-tuples $\bar{x} = x_0, x_1, \ldots, x_{n-1}$ so that $\cee \bar{x}$ is $\cee x_0 \cee x_1 \ldots \cee x_{n-1}$, and similarly for $\cee^*$. We note that the order is important here.

\begin{defn}\label{big}
If $1 \leq n < \omega$, an $n$-ary formula $\theta(x_0, x_1, \ldots, x_{n-1})$ is \textbf{big} if $T \vdash \cee \bar{x} \theta(\bar{x})$.
\end{defn}

The $1$-ary formula $x = x$ is big. The following lemma is a simple observation which allows us to extend big $n$-ary formulas to big $(n+1)$-ary formulas.

\begin{lem}\label{addone}
If $\theta(x_0, x_1, \ldots, x_{n-1})$ is a big $n$-ary formula and $i < n$, and $x^\prime$ is a new free variable, the formula $\theta(x_0, x_1, \ldots, x_i, x^\prime, x_{i+1}, \ldots, x_{n-1})$ is a big $(n+1)$-ary formula.
\end{lem}
\qed

\begin{defn}\label{hndl}
Suppose $1 \leq n < \omega$, $t(u, \bar{x})$ is an $(n+1)$-ary Skolem term and $\theta(\bar{x})$ is an $n$-ary formula. We say that $\theta(\bar{x})$ \textbf{handles} $t(\bar{x}, u)$ if:
\begin{displaymath}
\forall u \bigvee\limits_{i \leq n} \cee^* \bar{x} \cee^* \bar{y} [(\theta(\bar{x}) \wedge \theta(\bar{y})) \rightarrow (t(u, \bar{x}) = t(u, \bar{y}) \leftrightarrow \bigwedge\limits_{j < i} x_j = y_j )]
\end{displaymath}
\end{defn}

The following lemma states that every Skolem term can be handled.

\begin{lem}\label{handles}
If $\theta(\bar{x})$ is a big $n$-ary formula and $t(u, \bar{x})$ is an $(n+1)$-ary Skolem term, then there is a big $n$-ary formula $\theta^\prime(\bar{x})$ such that $T \vdash \forall \bar{x} [\theta^\prime(\bar{x}) \rightarrow \theta(\bar{x})]$ and $\theta^\prime(\bar{x})$ handles $t(u, \bar{x})$.
\end{lem}
\begin{proof}
We prove the lemma by induction on $n$. First suppose $n = 1$. Let $\mathcal{M}_0$ be the prime model of $T$ and $\mathcal{M}_0 \prec \mathcal{M}_1$ a superminimal conservative extension. Since $\theta(x)$ is big, there must be $c_0 \in M_1 \setminus M_0$ such that $\mathcal{M}_1 \models \theta(c_0)$.

Let $F = \{ \langle u, m \rangle : \mathcal{M}_2 \models t(u, c_0) = m \}$. By conservativity, $F \cap M_0 \in \Def(\mathcal{M}_0)$, so there is a partial $\mathcal{M}_0$-definable function $f$ such that $f(u) = m$ if and only if $\langle u, m \rangle \in F$ for all $u, m \in M_0$. Let $D$ be the domain of $f$, and let $X = \{ x : \mathcal{M}_0 \models \theta(x) \wedge \forall u \in D (t(u, x) = f(u)) \}$. This $X$ must be unbounded.

Enumerate $M_0 \setminus D$ as $u_0, u_1, \ldots$ (we can assume this set is infinite; if it is finite, the argument is similar). Let $x_0$ be the least $x \in X$. Given $x_0, \ldots, x_i$, let $x_{i+1}$ be the least $x \in X_0$ such that $x > x_i$ and

\begin{displaymath}
\forall j, k \leq i (t(u_j, x_k) \neq t(u_j, x)).
\end{displaymath}

Let $\theta^\prime(x)$ define the set $\{ x_i : i \in M_0 \}$. Then $\theta^\prime$ is big, since the set is unbounded, and handles the term $t(u, x)$.

Let $n > 0$ and assume that the lemma holds for $(n-1)$-ary big $\theta(\bar{x})$ and $n$-ary Skolem terms $t(u, \bar{x})$. Let $\theta(\bar{x})$ be a big $n$-ary formula and $t(u, \bar{x})$ an $(n+1)$-ary Skolem term. We again let $\mathcal{M}_0$ be the prime model of $T$ and let $\mathcal{M}_0 \prec \mathcal{M}_1 \prec \ldots \prec \mathcal{M}_n$ be a chain of superminimal, conservative extensions. Since $\theta(\bar{x})$ is big, there are $c_i \in M_{i+1} \setminus M_i$ such that $\mathcal{M}_n \models \theta(\bar{c})$.

We again let $F = \{ \langle u, m \rangle : \mathcal{M}_n \models t(u, \bar{c}) = m \}$, and by conservativity there is a partial $\mathcal{M}_0$-definable function $f(u, x_0, \ldots, x_{n-2})$ such that 
\begin{displaymath}
\mathcal{M}_{n-1} \models \forall u, m (f(u, c_0, \ldots, c_{n-2}) = m \iff \langle u, m \rangle \in F).
\end{displaymath}
Let $D = \{ u : \mathcal{M}_{n-1} \models \exists m (\langle u, m \rangle \in F) \}$, and by conservativity, $D \cap M_0$ is definable without parameters. We again call this set $D$.

By induction, we get that there is a big $n-1$-ary formula $\theta_0$ that handles $f$. Since $\scl(c_{n-1}) = \mathcal{M}_n$, we also have a Skolem term $g$ such that $\mathcal{M}_n \models g(c_{n-1}) = \langle c_0, c_1, \ldots, c_{n-2} \rangle$. Given $\bar{x}$ an $n-1$-tuple, we let $Y_{\bar{x}} = \{ x : \mathcal{M}_0 \models g(x) = \bar{x} \}$.

Similar to the above proof, we enumerate $M_0 \setminus D$ as $u_0, u_1, \ldots$ and we will construct a sequence $y_0, y_1, \ldots$ as follows. Enumerate those $(n-1)$-tuples $\bar{x}$ such that $\mathcal{M}_0 \models \theta_0(\bar{x})$ as $\bar{x}_0, \bar{x}_1, \ldots$, so that each such $\bar{x}$ appears infinitely often. Let $y_0$ be the least element of $Y_{\bar{x}_0}$. Given $y_0, \ldots, y_i$, let $y_{i+1}$ be the least $y > y_i$ such that $y \in Y_{\bar{x}_{i+1}}$ and $\forall j \leq i, k \leq i (t(u_j, \bar{x}_k, y_k) \neq t(u_j, \bar{x}_{i+1}, y))$. Let $X = \{ y_i : i \in M_0 \}$ and let $\theta^\prime(x_0, \ldots, x_{n-1})$ be the formula
\begin{displaymath}
\theta_0(x_0, \ldots, x_{n-2}) \wedge x_{n-1} \in X \wedge g(x_{n-1}) = \langle x_0, \ldots, x_{n-2} \rangle.
\end{displaymath}
Then $\theta^\prime$ is big and handles $t$.
\end{proof}

Let $\mathcal{M}_0 \models T$ be the prime model. Let $t_0(u, \bar{x}), t_1(u, \bar{x}), \ldots$ be an enumeration of all Skolem terms so that each $t_n$ has at most $(n+1)$ free variables. Let $s_0, s_1, \ldots$ be an enumeration of the (countably many) compact elements of $\gamma$. Given $n < \omega$, let $\pi_n$ be the permutation of $\{ 0, \ldots, n \}$ such that $s_{\pi_n(0)} < s_{\pi_n(1)} < \ldots < s_{\pi_n(n)}$. For an $(n+1)$-ary formula $\theta(\bar{x})$, by $\theta(\pi_n(\bar{x}))$ we mean $\theta(x_{\pi_n(0)}, \ldots, x_{\pi_n(n)})$. Similarly, for an $(n+2)$-ary Skolem term $t(u, \bar{x})$, by $t(u, \pi_n(\bar{x}))$ we mean $t(u, x_{\pi_n(0)}, \ldots, x_{\pi_n(n)})$.

Using Lemmas \ref{addone} and \ref{handles}, we construct a sequence of formulas $\theta_0(x_0), \theta_1(x_0, x_1) \ldots$ such that, for each $n < \omega$:
\begin{itemize}
\item $\theta_n(\pi_n(\bar{x}))$ is big and
\begin{displaymath}
T \vdash \forall \bar{x} (\theta_{n+1}(\bar{x}) \rightarrow \theta_n(\bar{x})),
\end{displaymath}
\item If $t_n$ is $(m+1)$-ary, then there is an $m$-ary formula $\theta(\bar{x})$ such that $\theta(\pi_{m-1}(\bar{x}))$ handles $t_n(u, \pi_{m-1}(\bar{x}))$ and
\begin{displaymath}
T \vdash \forall \bar{x} (\theta_n(\bar{x}) \rightarrow \theta(\bar{x}))
\end{displaymath}
\end{itemize}

The set $\{ \theta_n(\bar{x}) : n \in \omega \}$ determines a complete, consistent type: if $\theta(u, \bar{x})$ is any formula, then the corresponding Skolem term $t(u, \bar{x})$ (defined as $t(u, \bar{x}) = 0$ iff $\theta(u, \bar{x})$ and $t(u, \bar{x}) = 1$ otherwise) is handled at some stage $n$. Let $c_0, c_1, \ldots$ be elements realizing this type and let $\mathcal{M} \models T$ be generated by these elements. We claim that $\mathcal{M}$ is as desired.

First, we show that $\Lt(\mathcal{M}) \cong \gamma$. Let $s_i$ and $s_j$ be compact elements of $\gamma$. We show that $s_i \leq s_j \iff \scl(c_i) \preceq \scl(c_j)$. Suppose $s_i \leq s_j$. Let $m$ be the maximum of $i$ and $j$ and let $t(u, x_0, \ldots, x_m)$ be the term defined as $t(u, x_0, \ldots, x_m) = x_u$ if $u \leq m$ and $t(u, x_0, \ldots, x_m) = 0$ otherwise. There is $\theta(x_0, \ldots, x_m)$ such that $\theta(\pi_m(\bar{x}))$ handles $t(u, \pi_m(\bar{x}))$. Let $X$ be the set defined by $\theta$; that is,
\begin{displaymath}
X = \{ \langle x_0, \ldots x_m \rangle : \mathcal{M}_0 \models \theta(x_0, \ldots, x_m) \}.
\end{displaymath}
Then, for large enough $x$, if there are $\langle x_0, \ldots, x_{m-1} \rangle, \langle y_0, \ldots, y_{m-1} \rangle$ such that $\langle x_0, \ldots, x_{m-1}, x \rangle \in X$ and $\langle y_0, \ldots, y_{m-1}, x \rangle \in X$, then $x_k = y_k$ for each $k \leq m - 1$. Let $t_k(x)$ be the Skolem term defined (on all such large enough $x$) as that unique such $x_k$. Then, if $k$ is such that $\pi_m(k) = i$, $\mathcal{M} \models t_k(c_j) = c_i$.

Conversely, suppose $\scl(c_i) \preceq \scl(c_j)$. There is a term $f$ such that $\scl(c_j) \models f(c_j) = c_i$. Let $n$ be the maximum of $i$ and $j$. We claim that if $k_i$ and $k_j$ are such that $\pi_n(k_i) = i$ and $\pi_n(k_j) = j$, then $k_i \leq k_j$. To see this, let $t(u, \bar{x})= f(x_{k_j})$, and then suppose $\theta(\pi_n(\bar{x}))$ is big and handles $t(u, \pi_n(\bar{x}))$. Then if $k_j < k_i$, if we fix $x_0, \ldots, x_{k_j}$, there would be infinitely many different $x_{k_i}$ such that $f(x_{k_j}) = x_{k_i}$, which is impossible.

Next, we show that if $b \in M$, then there is some $c_i$ such that $\scl(b) = \scl(c_i)$. Let $n < \omega$, $m \in \mathcal{M}_0$, and $t(u, x_0, \ldots, x_{n-1})$ be such that 
\begin{displaymath}
\mathcal{M} \models t(m, \pi_{n-1}(\bar{c})) = b.
\end{displaymath}
Then there is some $\theta(\bar{x})$ such that $\theta(\pi_{n-1}(\bar{x}))$ handles $t(u, \pi_{n-1}(\bar{x}))$. Let $i \leq n$ be such that, in $\mathcal{M}_0$, the following statement holds:
\begin{align*}
\mathcal{M}_0 \models \cee^* \bar{x} \cee^* \bar{y} [\theta(\pi_{n-1}(\bar{x})) \wedge \theta(\pi_{n-1}(\bar{y})) \rightarrow (t(m, \pi_{n-1}(\bar{x})) = t(m, \pi_{n-1}(\bar{y})) \leftrightarrow \\ \bigwedge\limits_{j \leq i} x_{\pi_{n-1}(j)} = y_{\pi_{n-1}(j)})].
\end{align*}
Therefore, there are Skolem functions $f$ and $g$ so that $\mathcal{M} \models f(c_{\pi_{n-1}(0)}, \ldots, c_{\pi_{n-1}(i)}) = b$ and $\mathcal{M} \models g(b) = \langle c_{\pi_{n-1}(0)}, \ldots, c_{\pi_{n-1}(i)} \rangle$. Combining this with the argument above, we have that $\scl(b) = \scl(c_{\pi_{n-1}(i)})$.

This means that all the finitely generated substructures of $\mathcal{M}$ are the $\scl(c_i)$ for each $i < \omega$ and therefore that $\Lt(\mathcal{M}) \cong \gamma$.

Lastly, we show that $\scl(c_i) \prec_\text{cons} \mathcal{M}$ for each $i < \omega$. Let $X \subseteq M$ be defined as
\begin{displaymath}
X = \{ u : \mathcal{M} \models \phi(u, \pi_{n-1}(\bar{c})) \}.
\end{displaymath}
Then if $c_j$ is such that $c_0, \ldots, c_{n-1} \in \scl(c_j)$, clearly $X \cap \scl(c_j) \in \Def(\scl(c_j))$. Suppose $i$ is such that there are some $k < n$ such that $c_k \not \in \scl(c_i)$. Let $t(u, \bar{x})$ be the Skolem term such that $t(u, \bar{x}) = 0$ iff $\phi(u, \bar{x})$ and $t(u, \bar{x}) = 1$ otherwise. If $\theta(\pi_{n-1}(\bar{x}))$ is big and handles $t(u, \pi_{n-1}(\bar{x}))$, it must be the case that, for each $u$, 
\begin{displaymath}
\cee^* \bar{x} (\theta(\pi_{n-1}(\bar{x})) \rightarrow \phi(u, \pi_{n-1}(\bar{x}))) \vee \cee^* \bar{x} (\theta(\pi_{n-1}(\bar{x})) \rightarrow \lnot \phi(u, \pi_{n-1}(\bar{x})))
\end{displaymath}
Let $\theta$ be such that $\theta(\pi_{n-1}(\bar{x}))$ handles $t(u, \pi_{n-1}(\bar{x}))$, and let $c_{\pi_{n-1}(0)}, \ldots, c_{\pi_{n-1}(j)} \in \scl(c_i)$. Then $u \in X \cap \scl(c_i)$ if and only if
\begin{align*}
\scl(c_i) \models \cee^* x_{j+1} \cdots \cee^* x_{n-1} \: \theta(c_{\pi_{n-1}(0)}, \ldots, c_{\pi_{n-1}(j)}, x_{j+1}, \ldots, x_{n-1}) \\ \wedge \: \phi(u, c_{\pi_{n-1}(0)}, \ldots, c_{\pi_{n-1}(j)}, x_{j+1}, \ldots, x_{n-1}).
\end{align*}
\end{proof}

\begin{cor}
Suppose $\gamma$ is a countable algebraic linear order. Then there is an Enayat $\mathcal{M} \models \PA$ such that $\Lt(\mathcal{M}) \cong \gamma$.
\end{cor}\qed

\section{Open Problems}\label{opens}

From Theorem \ref{cuts}, if a model of $\PA$ is countable, has no proper cofinal submodel and is a conservative extension of each of its elementary cuts, then it is Enayat. Additionally, Lemma \ref{cofinality} shows that models with proper cofinal submodels cannot be Enayat. A negative answer to the following problem would complete the classification of Enayat models:

\begin{prob}
Suppose $\mathcal{M} \models \PA$ is countable but is not a conservative extension of (at least) one of its proper elementary cuts. Can $\mathcal{M}$ be Enayat?
\end{prob}

Corollary \ref{enayatlattices} characterizes the finite lattices which can appear as a substructure lattice of an Enayat model. We do not have such a characterization for countable lattices, though Corollary \ref{countablelattices} provides a first attempt.

\begin{prob}
Suppose $L$ is a countable lattice such that there is $\mathcal{M} \models \PA$ with $\Lt(\mathcal{M}) \cong L$. Under what circumstances is there an Enayat model $\mathcal{M}$ such that $\Lt(\mathcal{M}) \cong L$?
\end{prob}

\begingroup\csname @temptokena\endcsname{}
\theendnotes
\endgroup

\bibliographystyle{plain}
\bibliography{library}

\end{document}